\documentclass[12pt]{amsart}
\usepackage{geometry} 
\usepackage{amsmath,amscd,amssymb, amsthm, pifont}
\theoremstyle{plain}
\newtheorem{dfn}{Definition}
\newtheorem{thm}{Theorem}
\newtheorem{cor}{Corollary}
\newtheorem{lem}{Lemma}

\newcommand\N{\mathbb{N}}

\title{Extending A Theorem of Herstein}
\author{Cayley Pendergrass-Rice}
\date{} 
\thanks{This paper contains results from the author's dissertation, completed at the University of California, San Diego under the direction of Lance Small.  Thanks also to Efim Zelmanov for his insights and conversations.}

\begin{document}

\begin{abstract}
Just infinite algebras have been considered from various perspectives;  a common thread in these treatments is that the notion of just infinite is an extension of the notion of simple.  We reinforce this generalization by considering some well-known results of Herstein regarding simple rings and their Lie and Jordan structures and extend these results to their just infinite analogues.  In particular, we prove that if $A$ is a just infinite associative algebra, of characteristic not $2,3,$ or $5,$ then the Lie algebra $[A,A]/(Z\cap[A,A])$ is also just infinite (where $Z$ denotes the center of $A$).
\end{abstract}

\maketitle

\footnotetext{AMS classification 16, 17}

\section{Intro and Notation}
Herstein studied simple noncommutative rings and their Lie and Jordan structures (\cite{Her1}, \cite{Her2}).  Among many theorems describing the structure forced by requiring no two sided ideals, he published a series of papers, later included in \cite{HerTopics}, studying the related structures of the Jordan and Lie rings of simple associative rings. The following theorem concisely summarizes some of these results.

\begin{thm}(Herstein) \label{thm:Herstein}
If $A$ is a simple ring of characteristic not equal to $2$ then 
\renewcommand{\theenumi}{\roman{enumi}}
\begin{enumerate}
\item $A^+$ is a simple Jordan ring. \label{thm:Hersteinpt1}
\item $[A,A]/(Z\cap [A,A])$ is a simple Lie ring, where $Z$ is the center of $A$. \label{thm:HerRR}
\end{enumerate}
\end{thm}

Simple rings have been generalized in various directions.    One such generalization is to allow only a few, large ideals.  More formally, we insist that all non-zero two sided ideals have finite codimension.  

\begin{dfn}
An associative $k$-algebra $A$ is called \emph{just infinite dimensional}, or \emph{just infinite} for short, if $\dim_{k} A = \infty$ and $\dim_{k}A/I<\infty$ for all $(0)\neq I \triangleleft A$.
\end{dfn}

Just infinite algebras arise naturally.  For instance, $k[x]$, the algebra of polynomials in one variable over a field is just infinite, as are $k[x,x^{-1}]$, $k[[x]]$, and, vacuously, any simple algebra.  Beyond these, a class of examples not satisfying a polynomial identity are constructed in the language of algebraic geometry in \cite{rrz}.  The authors derive examples of $\N-$graded just infinite dimensional algebras, which they dub \"projectively simple\", from twisted homogenous coordinate rings.   Another interesting example is described in \cite{fs} and stems from the well-known Golod-Shafarevich example of an infinite dimensional algebra which is nil but not nilpotent.  Bartholdi generates another class of examples from groups acting on trees in \cite{bb}. 

Although the statement of Theorem \ref{thm:Herstein} as phrased above is nice for its brevity and parallel treatment of the Lie and Jordan cases, Theorem \ref{thm:Herstein}, part \ref{thm:Hersteinpt1} appears in an incarnation more convenient for our purposes in \cite{HerTopics}.

\begin{thm}(Herstein)
Let $A$ be a ring of characteristic not equal to 2.  Suppose that $A$ has no non-zero nilpotent ideals, i.e., $A$ is a semiprime ring.  Then any non-zero Jordan ideal of $A$ contains a non-zero associative ideal of $A$. 
\end{thm}

From this, it is clear that if $A$ is a simple algebra, so too is $A^+$.  Taken in conjunction with Theorem 1 of \cite{fp}, which proves all just infinite algebras are prime, we see the immediate implication for just infinite rings.

\begin{cor}
Let $A$ be a just infinite algebra over a field $k$ with $char(A) \neq 2$.  Then $A^+$ is also a just infinite algebra over $k$.
\end{cor}

This result and the parallelism in Theorem \ref{thm:Herstein} suggest we might extend Theorem \ref{thm:Herstein} part \ref{thm:HerRR} into the context of just infinite algebras also.  In fact, Herstein's result for Lie algebras does extend, though not as readily.  We should note that the broad strokes, if not the details, of this result follow Herstein's treatment of the simple case. 

In pursuing an analogue of Theorem \ref{thm:Herstein} for just infinite dimensional algebras, we first outline our notation.  The next section continues, then, with a sequence of inclusions necessary for the result and finally we prove the principal theorem:

\begin{thm} \label{thm:main}
If $A$ is a just infinite dimensional algebra without $2$, $3$, or $5-$torsion, then $[A,A]/(Z \cap[A,A])$ is also just infinite (as a Lie algebra).
\end{thm}
 
This paper contains results from the author's dissertation, completed at the University of California, San Diego under the direction of Lance Small.  Thanks also to Efim Zelmanov for his insights and conversations.

\section{Notation and Preliminary Inclusions}

We denote by $A^+$ the Jordan ring of an associative ring $A$ with multiplication written as $\ast$.  The Lie ring of $A$ will be written $A^-$ with product the commutator $ [ \cdot , \cdot ]$.  A Jordan ideal  is an ideal of the Jordan ring of $A$; in other words, an additive subgroup of $A$ closed under the operation $\ast$.  Similarly, we can define Lie ideals of $A$ as ideals of $A^-$.  By $[A,A]$ we mean the subalgebra of $A^-$ generated by all elements of the form $[x,y]$ for some $x,y \in A$.

For aesthetic reasons, in the case of repeated commutators we will omit the internal braces; thus $[[x,y],z]$ will instead be written $[x,y,z]$ and, more generally, $[x_1,x_2,...,x_n]$ will denote the commutator $\underbrace{[...[}_{n-1}x_1,x_2],x_3],...],x_n]$.  

For the remainder, $A$ is a $k-$algebra of characteristic not equal to 2 and $U$ is a Lie ideal of $[A,A]$.  Define the set $S(U) :=\{a \in A | [a,A] \subseteq U\}$.  It's clear from the definition that $S(U)$ is a Lie ideal of $A$.  We next list some additional properties of $S(U)$, which require the following technical lemma.  

\begin{lem} \label{lem:tech}
If $a,b \in A$, where $A$ is an associative algebra with Lie and Jordan products as defined above, then
\renewcommand{\theenumi}{\roman{enumi}}
\begin{enumerate}
\item $ab=1/2[a,b]+1/2a \ast b$. \label{lem:tech1}
\item $[a, b \ast c]=[a,b] \ast c + [a,c] \ast b$ \label{lem:tech2}
\end{enumerate}
\end{lem}

The details require some finicky manipulations, but little else.  With this, we can prove a few useful properties of $S(U)$.

\begin{lem} \label{lem:lem1} Let $U$ be a Lie ideal of $[A,A]$.  Then 
\renewcommand{\theenumi}{\roman{enumi}}
\begin{enumerate}
\item $S(U)$ is a subalgebra of A.
\item $[U,U] \in S(U)$. \label{lem:lem1b}
\item $[U,S(U)] \subset S(U)$.

\end{enumerate}
\end{lem}

\begin{proof} \ 
\renewcommand{\theenumi}{\roman{enumi}}
\begin{enumerate}
\item
To prove $S$ is closed under multiplication consider $[rs, A]=rsA-Ars$.  Using the technical lemma above, Lemma \ref{lem:tech} part \ref{lem:tech1}, 
\begin{eqnarray}
rsA-Ars=\{1/2[r,s]+1/2(r\ast s)\}A -A\{1/2[r,s]+1/2(r\ast s)\}. 
\end{eqnarray}

Noting that $[r,s]A=r[s,A]+[rA,s]$, we can reduce this to 
\begin{equation} 
rsA-Ars=1/2\{ [r,[s,A]]+[rA,s]-[As,r]\} + 1/2[(r\ast s),A]. \ \ \ \label{lem:eqn1}
\end{equation}
On the right hand side, the first group of terms is in $U$ because $r,s \in S$.  Turning to the last term of equation \ref{lem:eqn1}, we note that 
\begin{eqnarray}
[(r\ast s),A]= [r, s \ast A]+[s, r \ast A]. 
\end{eqnarray} 
Because $r \in S$, $[r,s \ast A] \subseteq U$.  Similarly we find that the second term, and hence $[(r\ast s),A]$ itself, is contained in $U$.  Thus, returning to equation \ref{lem:eqn1}, we see that $[rs,A] \subseteq U$ and so $rs \in S$.

A straightforward computation shows additive closure; the remaining conditions to prove $S$ is a subalgebra of $A$ also are shown easily. 
\item
Let $x,y \in U$ and $a \in A$.  Then, by the Jacobi identity, $[[x,y],a]=[x,[y,a]]+[y,[a,x]]$.  As $x \in U, [y,a] \in [A,A]$, and $U$ is a Lie ideal of $[A,A]$, $[x,[y,a]] \in U$.  Similarly, $[y,[a,x]] \in U$ and so $[[U,U], A] \subseteq U$ and hence we have $[U,U] \subseteq S$ as desired.

\item $[U, S(U)] \subset U$ by the definition of $S(U)$.  Thus $[U,S(U),U] \subset [U,U] \subset U$, proving the statement. 
 
\end{enumerate}
\end{proof}


Given a Lie ideal, $U \triangleleft A^-$, denote by $U^{[i]}$ a member of the descending chain $U^{[1]} \supseteq U^{[2]} \supseteq U^{[n]} \supseteq \dots$ defined recursively by $U^{[1]}=U, U^{[n]}=[U^{[n-1]},U^{[n-1]}]$.  Although similar, note that this is not the same as the lower central series whose elements are denoted $U^{(n)}$.  
Consider, then, $[S^{[k]},[A,A]]$; from Lemma \ref{lem:lem1} it follows that this subring is contained back in $S^{[k]}$, i.e. $[S^{[k]},[A,A]] \subseteq S^{[k]}$.  Subsequently,  we can show that \begin{equation} [S^{[k+1]},A] \subseteq [[S^{[k]},S^{[k]}],A] \subseteq [S^{[k]},[A,A]] \subseteq S^{[k]}. \end{equation} \label{inc1}
Then by Lemma \ref{lem:tech}, part \ref{lem:tech1}, 
\begin{equation} 
S^{[3]}  \subseteq [S^{[3]}, A] + S^{[3]} \ast A \subseteq S^{[2]} + S^{[3]} \ast A,
\end{equation} 

and, applying Lemma \ref{lem:tech}, part \ref{lem:tech2},  we have
\begin{equation}
S^{[3]}A= [S^{[2]},S^{[2]}] \ast A \subseteq [S^{[2]},S^{[2]} \ast A] + [S^{[2]}, A] \ast S^{[2]}.
\end{equation}
By equation \ref{inc1}, above, and  Lemma \ref{lem:lem1} the right hand side of this inclusion, and thus also the left, is contained in $S$.  Thus we have demonstrated that $S^{[3]}A \subseteq S$.  Using an identical argument on the other side, we can conclude that $AS^{[3]} \subseteq S$, too.  These together prove that the ideal generated by $S^{[4]}$ in $A$, $id_A(S^{[4]})=A[S^{[3]},S^{[3]}]A,$ is also contained in $S$.

\begin{lem} \label{lem:idS4}
$id_A(S^{[4]}) \subseteq S$.
\end{lem}

\section{And Now For Something Completely Different}
With these preliminary computations complete we are ready to attend to Theorem \ref{thm:main}.  Although less elegant than the theorem itself, the following lemma is at the heart of that result.  

\begin{lem} \label{lem:UinZ}
Let $A$ be a semiprime ring without $2,3, or 5$ torsion.  Let $U$ be an abelian subgroup of $A$ with $[U,[A,A]] \subseteq U$ and also $[U,U] \subseteq Z$, where $Z$ is the center of $A$.  Then $U \subseteq Z$.
\end{lem}

\begin{proof}
Let $a,b,c \in A$.  Then, by a straightforward induction,
\begin{equation} [ab,\underbrace{c,c...,c}_n] = \sum_{i \leq n} \binom{n}{i}[a,\underbrace{c,...,c}_i] [b,\underbrace{c,...,c}_{n-i}].
\end{equation}

First, suppose that \begin{equation}[A,\underbrace{c,...,c}_{n+1}]=(0) \end{equation} and $A$ does not have $\binom{2n}{n}-$torsion.  Then for all $a \in A$, 
\begin{equation}
[a^2, \underbrace{c,...,c}_{2n}]=\binom{2n}{n}[a,\underbrace{c,...,c}_{n}]^2 + \sum_{\underset{i+j= 2n} {max\{i,j\} \geq n+1}} [a,\underbrace{c,...,c}_{i}][x,\underbrace{c,...,c}_{j}].
\end{equation}
Thus for any $a \in A$, \begin{equation} [x,\underbrace{c,...,c}_{n}]^2=0.\end{equation}

Consider, now, $x\in U$.  
\begin{equation}
[A,x,x,x,x] \subseteq [A,A,x,x,x] \subseteq [U,U,U] \subseteq [Z,U] =(0).
\end{equation}
Because
\begin{equation}
[a,x,x,x] \in [A,A,x,x] \subseteq [U,U] \subseteq Z
\end{equation}
and, by the comment above, \begin{equation} [a,\underbrace{x,...,x}_{n}]^2 =0 \end{equation} for all $a \in A$, we see that  $([a,x,x,x]A)^2=(0)$. As $A$ is a semiprime algebra, this implies that the ideal itself is zero.  Thus for all $a \in A$,  we have that $([a,x,x,x]A)=(0)$ which, in turn, implies that $[a,x,x,x]=(0)$ and hence we have $[A,x,x,x] =(0)$.  Iterating the strategy with $[A,x,x,x] =(0)$, we have $[a,x,x]=0$ for all $a\in A$ and $[A,A,x,x]=(0).$

Now let $x$ be a nilpotent element of $U$, with index of nilpotence $n$. That is, $x^n=0, x^{n-1} \neq 0$.  Then for $a \in [A,A]$, 
\begin{equation}
xax=\frac{1}{2}([x,a,x]-x^2a-ax^2)=-\frac{1}{2}(x^2a+ax^2)=0.
\end{equation}
Thus $x^{n-1}[A,A]x^{n-1}=0$.  Then given $a\in A$, $[ax^{n-1},a] \in [A,A]$ so that 
\begin{eqnarray}
x^{n-1}[ax^{n-1},a]x^{n-1} 
&=&x^{n-1}(ax^{n-1}a-a^2x^{n-1})x^{n-1}\\
&=&x^{n-1}ax^{n-1}ax^{n-1}=0
\end{eqnarray}
which implies that $(x^{n-1}A)^3=(0)$.  Because $A$ is semiprime, $x^{n-1}A=(0)$ and thus $x^{n-1}=0$.  This contradicts the assumption that $x$ has index of nilpotence $n$, so $U$ must not contain any nonzero nilpotent elements.

We know that, for all $a \in A$,  $[a,x,x]^2 =0$. This, together with $U$ not containing nonzero nilpotent elements, implies that $[A,x,x] =(0)$. This, coupled with the beginning of the proof, shows that $[a,x]^2=0$ for all $a\in A$.  Now, if $a \in [A,A]$ then $[a,x] \in U$ and $U$ has no nonzero nilpotent elements; hence $[a,x]=0$.  Thus $[A,A,x]=0$ and $[A,A,U]=0$.  

Since $x$ commutes with every element of $[A,A]$, we have $x[x,ab]=[x,ab]x$ for any $a,b \in A$.  Expanding this using the identity $[x,ab]=[x,a]b+a[x,b]$ and noting that $x$ commutes with $[x,a]b+a[x,b]$, $[x,a]$, and $[x,b]$ gives that $2[x,a][x,b]=0$ and thus that $[x,a][x,b]=0$ for all $a,b \in A$.  Now setting $b=ab$ we have $[x,ab][x,a]=0$.  Then,
\begin{eqnarray}
[x,ab][x,a]&=&[x,ab][x,a] -a[x,b][x,a] \\
&=&([x,ab]+a[b,x])[x,a]\\
&=&[x,a]b[x,a]. \label{eq:uxyux}
\end{eqnarray}
Because \ref{eq:uxyux} holds for all $b \in A$, $[x,a]A[x,a] =(0)$.  $A$ is semi-prime, so $[x,a]=0$ from which we conclude that $x \in Z$, and, in turn, $U \subset Z$.

\end{proof}

With this in hand we are ready to prove the principal theorem. 
\begin{thm} 
If $A$ is a just infinite dimensional algebra without $2$, $3$, or $5-$torsion, then $[A,A]/(Z [A,A])$ is also just infinite.
\end{thm}
\begin{proof}
Let $U$ be a nonzero Lie ideal of $[A,A]$. Suppose $S^{[4]} \neq (0)$.  Then by \ref{lem:idS4}, $S$ is of finite codimension in $A$, implying that $A= V+S$ and $dimV < \infty$.  $[A,A] =[S+V,S+V] \subseteq [S,A]+[V,V] \subseteq U+[V,V]$.  As $dimV < \infty$, $U$ is of finite codimension in $A$.

Now suppose $S^{[4]}=(0)$.  $A$ is just infinite hence prime and applying Lemma \ref{lem:UinZ} to $S^{[3]}$ shows $S^{[3]} \subseteq Z$.  Again, using Lemma \ref{lem:UinZ} on $S^{[2]}$, we get $S^{[2]} \subseteq Z$.  Repeating once more with $S$ gives $S\subseteq Z$.  However, by Lemma \ref{lem:lem1}, we know that $[U,U] \subseteq S \subseteq Z$ and so by Lemma \ref{lem:UinZ} we have $U \subseteq Z$. 
\end{proof}

From here, the next task ought to be a careful consideration of the excluded cases of Theorem \ref{thm:main}: algebras of characteristic $2$,$3$, and $5$. Herstein's results were strengthened by Baxter in  \cite{baxter} where he shows that the only exceptions are the algebras of $2\times2$ matrices over fields of characteristic 2.  We hope to  describe the exceptions to the conclusion of Theorem \ref{thm:main} specifically as well. We conjecture that the result holds in characteristic $3$ and $5$, but a quick inspection of the above proof shows that these characteristics require a different approach. Additionally, results of Jacobson and Rickart \cite{jr}, Baxter \cite{baxter}, and Amitsur \cite{amitsur}, stemming from Theorem \ref{thm:Herstein} (all nicely summarized in \cite{Her4}) should be considered in the context of just infinite algebras as they may have corresponding extensions.

\end{document}